\documentclass[12pt,reqno]{amsart}
\usepackage{amsmath, amsthm, amssymb,booktabs}

\advance \topmargin by -\headheight
\advance \topmargin by -\headsep
     
\evensidemargin \oddsidemargin
\newtheorem{theorem}{Theorem}
\newtheorem{lemma}[theorem]{Lemma}

\theoremstyle{definition}

\theoremstyle{remark}

\def\le{\leqslant} \def\ge{\geqslant}

\begin{document}
\title[Finite abelian groups via congruences]{Finite abelian groups via congruences}
\author[Trevor D. Wooley]{Trevor D. Wooley}
\address{Department of Mathematics, Purdue University, 150 N. University Street, West Lafayette, IN 47907-2067, USA}
\email{twooley@purdue.edu}

\maketitle

\begin{abstract}
For every finite abelian group $G$, there are positive integers $n$ and $d$ such that $G$ is 
isomorphic to the multiplicative group of $d$-th powers of reduced residues modulo $n$.
\end{abstract}

\bigskip

As every beginning student in number theory learns, when $n\in \mathbb N$, the integers 
$a$ with $1\le a\le n$ and $(a,n)=1$ form an abelian group when equipped with the binary 
operation defined by multiplication modulo $n$. For each positive integer $d$, this 
multiplicative group $U_n$ of reduced residues modulo $n$ contains the subgroup of $d$-th 
powers, namely $U_n^{(d)}=\{ a^d:a\in U_n\}$. The goal of this note is to show that, in 
fact, every finite abelian group is isomorphic to one of the latter shape.

\begin{theorem} Let $G$ be a finite abelian group. Then there exist 
positive integers $n$ and $d$ having the property that $G\cong U_n^{(d)}$.
\end{theorem}

Standard introductions to algebra will describe the fundamental theorem of finitely 
generated abelian groups (see \cite[Theorem 3 of \S5.2]{DF2004}). Thus, writing 
$\mathbb Z_r$ for the additive group of integers modulo $r$, each finite abelian group 
$G$ is isomorphic to a direct product of the shape 
$\mathbb Z_{m_1}\times \mathbb Z_{m_2}\times \cdots \times \mathbb Z_{m_k}$, where 
the integers $m_1,\ldots ,m_k$ satisfy
\begin{equation}\label{1}
m_i\ge 2\quad \text{and}\quad m_{i+1}|m_i\quad (1\le i<k).
\end{equation}
Given distinct odd prime numbers $p_i$ with $p_i\equiv 1\, (\text{mod}\ m_i)$, say with 
$p_i=1+m_id_i$, one finds that $\mathbb Z_{m_i}$ is isomorphic to the group 
$U_{p_i}^{(d_i)}=\langle g_i^{d_i}\rangle$, where $g_i$ is any 
primitive root modulo $p_i$. Writing $n=p_1p_2\cdots p_k$, the Chinese Remainder 
Theorem delivers the familiar conclusion that $G$ is isomorphic to a subgroup of 
$U_n$. Thus, one has
\[
G\cong U_{p_1}^{(d_1)}\times U_{p_2}^{(d_2)}\times \cdots \times 
U_{p_k}^{(d_k)}.
\]
Although it is comforting to note that every finite abelian group is isomorphic to a subgroup 
of the multiplicative group of residues modulo $n$, for some $n\in \mathbb N$, the need to 
work with a collection of exponents $d_1,\ldots ,d_k$ corresponding to the prime divisors of 
$n$ is somewhat inelegant. The primary motivation for establishing our theorem is to 
provide a self-contained description of the abelian group $G$ with only a pair of integers. 

\par Our goal, of identifying a simple congruence-based realization of the abelian group 
$G$ of the shape $U_n^{(d)}$, would follow were one able to find distinct primes 
$p_1,\ldots ,p_k$ corresponding to a common value $d$ for the $d_i$. Indeed, it would 
then follow as a consequence of the Chinese Remainder Theorem that 
\[
G\cong U_{p_1}^{(d)}\times U_{p_2}^{(d)}\times \cdots \times U_{p_k}^{(d)}\cong 
U_n^{(d)}.
\]
This approach relies on a special case of Dickson's conjecture (see \cite{Dic1904}). When 
$k\ge 1$, this as yet unproven conjecture asserts that given $a_i\in \mathbb Z$ and 
$m_i\in \mathbb N$, there are infinitely many positive integers $d$ for which the 
$k$-tuple $(m_1d+a_1,\ldots ,m_kd+a_k)$ consists of prime numbers, unless there is a 
congruence condition preventing such from occurring. In the case presently of interest to 
us, in which $a_i=1$ for each $i$, congruence obstructions are absent, though the 
conjecture has been established only in the case $k=1$ (a special case of Dirichlet's theorem, for which 
see \cite[Corollary 4.10]{MV2007}). The challenge of obtaining an 
unconditional conclusion requires careful selection of arithmetic progressions in which to 
search for the primes $p_i$.

\begin{lemma} Suppose that the integers $m_1,\ldots ,m_k$ satisfy the condition 
\eqref{1}. Then there are distinct prime numbers $p_1,\ldots ,p_k$ satisfying the 
congruences 
\[
p_i\equiv 1+m_1m_i\ (\text{mod}\, {m_1^2m_i})\quad (1\le i\le k).
\]
\end{lemma}   

\begin{proof} Dirichlet's theorem on prime numbers in arithmetic progressions (see 
\cite[Corollary 4.10]{MV2007}) shows that for each $i$, there are infinitely many primes 
$p$ with $p\equiv 1+m_1m_i\ (\text{mod}\, {m_1^2m_i})$, since it is evident that 
$(1+m_1m_i,m_1^2m_i)=1$. The desired conclusion is immediate.
\end{proof}

\begin{proof}[The proof of the theorem] Let $G$ be a finite abelian group, whence there 
exist integers $m_1,\ldots ,m_k$ satisfying the condition (\ref{1}) for which 
$G\cong \mathbb Z_{m_1}\times \mathbb Z_{m_2}\times \cdots \times 
\mathbb Z_{m_k}$. Given distinct prime numbers $p_1,\ldots ,p_k$ supplied by the lemma, 
define the integers $u_i$ via the relation $p_i-1=m_1m_i(1+m_1u_i)$, and then write 
$y_i=1+m_1u_i$. Also, put $d=m_1y_1\cdots y_k$. Since $p_i-1=m_1m_iy_i$, we see that 
for $1\le i\le k$, one has $(p_i-1,d)=m_1y_iD_i$, where $D_i$ is the greatest common 
divisor of $m_i$ and
\[
\prod_{\substack{1\le j\le k\\ j\ne i}}(1+m_1u_j).
\]
Since $m_i|m_1$, this greatest common divisor is plainly $1$, and so $(p_i-1,d)=m_1y_i$, 
yielding $(p_i-1)/(p_i-1,d)=m_i$. Thus 
$U_{p_i}^{(d)}\cong \mathbb Z_{m_i}$ $(1\le i\le k)$. Consequently, if we put 
$n=p_1p_2\cdots p_k$, then it follows from the Chinese Remainder Theorem that
\[
\mathbb Z_{m_1}\times \mathbb Z_{m_2}\times \cdots \times \mathbb Z_{m_k}
\cong U_{p_1}^{(d)}\times U_{p_2}^{(d)}\times \cdots \times 
U_{p_k}^{(d)}\cong U_n^{(d)},
\]
whence $G\cong U_n^{(d)}$. This completes the proof of the theorem.
\end{proof}

The proof of our theorem is more or less algorithmic, in the sense that it shows how to 
determine the integers $n$ and $d$ from a standard presentation of the finite abelian 
group $G$. Bounds on these integers may be given in terms of the order $|G|$ of the 
abelian group $G$ of interest by using technology associated with Linnik's theorem 
\cite{Lin1944a, Lin1944b}. The latter shows that there is a positive number $L$ having the 
property that when $m$ and $a$ are integers with $m\ge 2$ and $(a,m)=1$, then the 
smallest prime number $p$ satisfying $p\equiv a\ (\text{mod}\, m)$ satisfies $p\le m^L$. 
In particular, Xylouris \cite{Xyl2011} has shown that such a prime exists with 
$p\le Cm^{5.18}$, for a suitable positive constant $C$. By employing such results, it would 
be possible to show that both $n$ and $d$ can be taken no larger than 
$\exp (c (\log |G|)^2)$, for a suitable positive constant $c$, though less profligate bounds 
might well be accessible.\par

Plucking the example of the abelian group $G=\mathbb Z_3\times \mathbb Z_{42}$ almost 
from thin air, we may illustrate the relative inefficiency of the method underlying our 
theorem. The smallest positive integer $d$ satisfying the property that $3d+1$ and 
$42d+1$ are simultaneously prime is $d=10$, so that 
\[
\mathbb Z_3\times \mathbb Z_{42}\cong U_{31}^{(10)}\times U_{421}^{(10)}\cong 
U_{13051}^{(10)}.
\]
Smaller realizations are given by 
\[
\mathbb Z_3\times \mathbb Z_{42}\cong U_{7}^{(8)}\times U_{337}^{(8)}\cong 
U_{2359}^{(8)}
\]
and
\[
\mathbb Z_3\times \mathbb Z_{42}\cong \mathbb Z_6\times \mathbb Z_{21}
\cong U_{13}^{(2)}\times U_{43}^{(2)}\cong U_{559}^{(2)}.
\]
Meanwhile, the approach suggested by the proof of our theorem asks that we seek  one 
prime congruent to $1+3\cdot 42=127$ modulo $3\cdot 42^2$, and a second congruent to 
$1+42^2$ modulo $42^3$. We find that $127$ and $1+42^2+42^3=75853$ are the 
smallest such primes. We then take $n=127\cdot 75853=9633331$ and 
$d=42\cdot 1\cdot 43=1806$. Thus, the realization of $G$ suggested by the proof of our 
theorem is
\[
\mathbb Z_3\times \mathbb Z_{42}\cong U_{127}^{(1806)}\times U_{75853}^{(1806)}
\cong U_{9633331}^{(1806)}.
\]
\medskip

\noindent {\bf Acknowledgment:} The author's work is supported by NSF grants 
DMS-1854398 and DMS-2001549. The author is grateful to the referees for their valuable 
suggestions and comments.

\end{document}